\documentclass[winfonts,hyperref,a4paper]{article}
\usepackage{mathrsfs}
\usepackage{amscd}
\usepackage{amsmath}
\usepackage{amssymb}
\usepackage{amsthm}
\usepackage{graphicx,tikz,enumerate}
\usepackage[left=2.54cm,right=2.54cm,top=3.17cm,bottom=3.17cm]{geometry}
\usepackage{listings}\usepackage[colorlinks, citecolor=red]{hyperref}
\theoremstyle{plain}
\newtheorem{defn}{Definition}
\newtheorem{theorem}{Theorem}
\newtheorem{lemma}{Lemma}

\newtheorem{remark}{Remark}

\newcommand{\T}{\ensuremath{\mathscr{T}}}
\newcommand{\p}{\ensuremath{\mathscr{P}}}
\newcommand{\F}{\ensuremath{\mathscr{F}}}

\begin{document}
\title{Connected door spaces and topological solutions of equations}

\author{Jianfeng Wu\footnotemark[1],
\and Chunli Wang\footnotemark[2],
\and Dong Zhang\footnotemark[3]}
\renewcommand{\thefootnote}{\fnsymbol{footnote}}
\footnotetext[1]{School of Mathematics and Statistics, Zhengzhou University, Zhengzhou, Henan, China. E-mail: wuj@zzu.edu.cn}
\footnotetext[2]{Tianjin Medical Laboratory, BGI, Tianjin, China.  E-mail: wangchunlimath@163.com}
\footnotetext[3]{School of Mathematical Sciences, Peking University, Beijing,  China.  E-mail: 13699289001@163.com}
\date{}

\maketitle
\begin{abstract}
The connected door space is an enigmatic topological space in which every proper nonempty subset is either open or closed, but not both. This paper provides an elementary proof of the classification theorem of connected door spaces. More importantly, we show that connected door topologies can be viewed as solutions of the valuation $f(A)+f(B)=f(A\cup B)+f(A\cap B)$ and the equation  $f(A)+f(B)=f(A\cup B)$, respectively. In addition, some special solutions, which can be regarded as a union of connected door spaces, are provided.
\end{abstract}

A topological space is called a {\sl door space} if every subset is either open or closed (or both).
In his classic textbook \cite{kelley}, John L. Kelly introduced this concept. He merely gave two easy facts about Hausdorff door spaces there. Beyond a few literatures on various types of door spaces, such as \cite{BDE2011,DGKM1998,D1995,DLT2013,OH2004}, we find the {\sl connected door space} has certain interesting connections with functions. For a connected door space $(X,\T)$, we may define a map $f:\p(X)\setminus\{\varnothing,X\}\to\{0,1\}$ by $f(A)=1$ if $A\in\T$, and $f(A)=0$ if $A^c\in\T$, where $A^c=X\setminus A$, and $\p(X)$ is
the {\sl power set} of $X$, i.e. $\p(X)=\{A:A\subset X\}$. For most topological spaces, such an $f$ fails to be a map. Since two-valued maps play important roles in mathematics,
especially in mathematical logic, we believe that more attentions should be paid on connected door spaces, or even door spaces.

In 1973, B. Muckenhoupt and V. Williams \cite{MW1973} showed that the number of connected door topologies on an infinite set $X$ with $|X|=k$ is $2^{2^k}$, where $|X|$ is the
cardinality of $X$.
In 1987, S. D. McCartan \cite{M1987} classified door spaces,  and showed the following classification theorem of connected door spaces.

\begin{theorem}
 \label{thm:classes-connected-door}
There are exactly three types of connected door spaces $(X,\T)$:
\begin{enumerate}[(1)]
\item $\T=\{A\subset X:a\not\in A\}\cup \{X\}$ for some $a\in X$ ({\sl excluded point topology});
\item $\T=\{A\subset X:a\in A\}\cup \{\varnothing\}$ for some $a\in X$ ({\sl included point topology});
\item $\T\setminus\{\varnothing\}$ is a free ultrafilter on $X$ (and $X$ is necessarily infinite in this case).
\end{enumerate}
\end{theorem}
We first give an elementary proof of Theorem    \ref{thm:classes-connected-door} via several auxiliary lemmas which also possess independent interest (see Section \ref{sec:1}). Then, we find that connected door topologies could be induced respectively (see Theorems \ref{thm:two-valued-add} and \ref{thm:two-valued} for details) by the   
set equations:
\begin{equation}\label{eq:additivity}
f(A)+f(B)=f(A\cup B)
\end{equation}
and
\begin{equation}\label{eq}
f(A)+f(B)=f(A\cup B)+f(A\cap B).
\end{equation}

\begin{theorem}\label{thm:two-valued-add}
Assume that $X$ has at least three points and $z_1\ne z_2$ are two complex numbers. Let $f:\p(X)\setminus\{\varnothing\}\to \{z_1,z_2\}$ be a surjection  satisfying \eqref{eq:additivity} for any $A,B\in\p(X)\setminus\{\varnothing\}$ with $A\cap B=\varnothing$. Then $0\in \{z_1,z_2\}$, $f^{-1}(f(X))$ is an ultrafilter, and $\{\varnothing\}\cup f^{-1}(f(X))$ is a connected door space.
\end{theorem}

Theorems \ref{thm:three-valued} and \ref{thm:three-valued-add} provide all topological solutions of the valuation \eqref{eq} and the equation \eqref{eq:additivity} in more complex setting.  Together with Theorem \ref{thm:two-valued-add},  connected door spaces could be viewed as some solutions to \eqref{eq:additivity} and \eqref{eq}. Proofs and detailed discussions are presented in  Section \ref{sec:2}.
These results suggest that certain topologies might be solutions of some set equations, which open a `door' to study topologies through equations.

\begin{theorem}\label{thm:three-valued}
For a topology $\T$ on $X$, there exists a surjection
 $f:\p(X)\to \{-1,0,1\}$ satisfying \eqref{eq} such that $\T=\{\varnothing,X\}\cup f^{-1}(1)$ or $\T=\{\varnothing,X\}\cup f^{-1}(-1)$, if and only if $\T$ is one of the following form:
  \begin{enumerate}[{Form} 1.]
  \item There exist $a\in X$ and an ultrafilter $\F'$ on $X\setminus\{a\}$ such that $\T=\{\varnothing,X\}\cup \F'$ or $\T=\{\varnothing,X\}\cup\{\{a\}\cup F:F\in \F'\}$;
  \item There exist $A\subset X$ and two free ultrafilters $\F'$ and $\F''$ on $A$ and $X\setminus A$, respectively, such that $\T=\{\varnothing,X\}\cup \{F'\cup F'':F'\in \F', F''\in\F''\}$;
  \item  $\T=\{\varnothing,X\}\cup\p(X\setminus\{a,b\})$ where $a\ne b\in X$.
  \end{enumerate}
\end{theorem}

\begin{theorem}\label{thm:three-valued-add}
Assume that $X$ has at least four points and let $\{z_1,z_2,z_3\}$ be a set of three different complex numbers. Let $f:\p(X)\setminus\{\varnothing\}\to \{z_1,z_2,z_3\}$ be a surjection satisfying \eqref{eq:additivity} for any $A,B\in\p(X)\setminus\{\varnothing\}$ with $A\cap B=\varnothing$.
Then $\{z_1,z_2,z_3\}=\{-z,0,z\}$ or $\{0,z,2z\}$ for some complex $z\in \mathbb{C}\setminus\{0\}$. Furthermore,
\begin{enumerate}[({T}1)]
\item
for the case of $\{z_1,z_2,z_3\}=\{-z,0,z\}$,  there is a topology $\T$ satisfying $\T=\{\varnothing,X\}\cup f^{-1}(-z)$ or $\T=\{\varnothing,X\}\cup f^{-1}(z)$ if and only if  there exists $p\in X$ and an ultrafilter $\F'$ on $X\setminus\{p\}$ such that $\T=\{\varnothing,X\}\cup \F'$;
\item
for the case of  $\{z_1,z_2,z_3\}=\{0,z,2z\}$, there is a topology $\T$ satisfying $\T=\{\varnothing,X\}\cup f^{-1}(2z)$ if and only if there exist $A\subset X$ and two  ultrafilters $\F'$ and $\F''$ on $A$ and $X\setminus A$, respectively, such that $\T=\{\varnothing,X\}\cup \{F'\cup F'':F'\in \F', F''\in\F''\}$.
\end{enumerate}
\end{theorem}

\section{An elementary proof of Theorem    \ref{thm:classes-connected-door}}
\label{sec:1}

In this section, a direct yet elementary proof of Theorem \ref{thm:classes-connected-door} is given. The idea comes from the following easy observations (which surely reveal more inner
structures of connected door spaces).

\begin{lemma}\label{lem:1}
Let $X$ be a connected door space. Let $A$ and $B$ be two disjoint nonempty subsets of $X$. If $A$ is open and $B$ is closed, then for each $C\subset X-(A\cup B)$ (see Fig.~\ref{fig:1open1closed}), $A\cup C$ is open and $B\cup C$ is closed.
\end{lemma}

\begin{figure}[h!]
\centering
\begin{tikzpicture}[scale=0.7]
 \draw (0,0) circle (3);
 \draw[style=dashed] (-1.3,1.3) circle (1);
 \draw (-2,-2)--(-2,-0.6)--(-0.6,-0.6)--(-0.6,-2)--(-2,-2);
 \draw (1.3,0) circle (1.3);
 \node (2) at (-1.3,1.3) {open};
 \node (4) at (-1.3,-1.3) {closed};
 \node (C) at (1.3,0) {$C$};
\end{tikzpicture}
\caption{If this picture appears in a connected door space (see Lemma \ref{lem:1}), then  the union of $C$ and the closed set is closed, while the union of $C$ and the open set is open.}
\label{fig:1open1closed}
\end{figure}
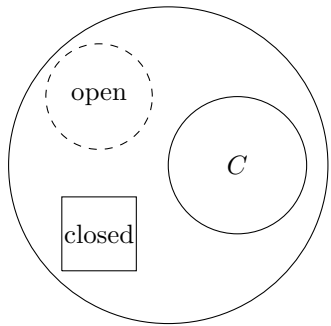

\begin{proof}
Let $D=X-(A\cup B\cup C)$. If $A\cup C$ is closed while $A\cup B$ is open, then $B\cup D=X-(A\cup C)$ is open, and hence $B=(A\cup B)\cap(B\cup D)$ is open, a contradiction to the assumption. If $A\cup C$ and $A\cup B$ are both closed, then $A=(A\cup B)\cap(A\cup C)$ is closed, again a contradiction to the assumption. Therefore, $A\cup C$ must be open. Similarly, $B\cup C$ must be closed.
\end{proof}

\begin{defn}
A topological space $X$ satisfies {\sl OCC} (abbreviation for `the open-closed condition') if there are no pairwise disjoint nonempty subsets $A,B,C$ and $D$ such that $A$ and $B$ are open while $C$ and $D$ are closed.
\end{defn}

\begin{lemma}\label{lem:2}
If $X$ is a connected door space, then $X$ satisfies {\rm OCC}.
\end{lemma}

\begin{proof}
Suppose there exist four pairwise disjoint nonempty subsets $A,B,C$ and $D$ such that $A$ and $B$ are open while $C$ and $D$ are closed. Applying Lemma~\ref{lem:1}, since $A$ is open and $D$ is closed, we may immediately conclude that $A\cup C$ is open. Applying Lemma~\ref{lem:1} again, since $C$ is closed and $B$ is open, $C\cup A$ must be closed. Thus, the proper nonempty subset $A\cup C$ is both open and closed, which is impossible.
\end{proof}

\begin{figure}[h!]
\centering
\begin{tikzpicture}[scale=0.7]
 \draw (0,0) circle (3);
 \draw[style=dashed] (1.3,1.3) circle (1);
 \draw[style=dashed] (-1.3,1.3) circle (1);
 \draw (2,-2)--(2,-0.6)--(0.6,-0.6)--(0.6,-2)--(2,-2);
 \draw (-2,-2)--(-2,-0.6)--(-0.6,-0.6)--(-0.6,-2)--(-2,-2);
 \node (1) at (1.3,1.3) {open};
 \node (2) at (-1.3,1.3) {open};
 \node (3) at (1.3,-1.3) {closed};
 \node (4) at (-1.3,-1.3) {closed};
\end{tikzpicture}
\caption{This picture cannot appear in a connected door space (see Lemma \ref{lem:2}), i.e., a connected door space does not admit four pairwise
disjoint nonempty subsets with two of them open and the rest closed.}
\label{fig:2open2closed}
\end{figure}
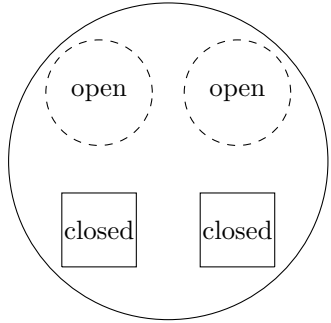

Here comes the proof of Theorem \ref{thm:classes-connected-door}. In fact, we shall prove: if a door space, with more than three elements, satisfies OCC,
then it must be of one of the three types shown in Theorem \ref{thm:classes-connected-door}.

\begin{proof}
Since all three cases presented in Theorem \ref{thm:classes-connected-door} are obviously connected door spaces, we only need to prove the converse.

If $|X|\ge 2$ and $X$ is a connected door space, then it is not discrete and thus there exists at least one closed singleton. Using Lemma \ref{lem:2}, $X$ has at least two open points if and only if it has exactly one closed singleton. Consequently, we only need to consider the following three cases.

\begin{enumerate}[(1)]
\item There is exactly one closed singleton in $(X,\T)$.

Let $\{a\}$ be the only closed singleton, and $a\in A\subsetneqq X$. If $|A^c|=1$, then $A^c$ is open. If $|A^c|>1$, take $b\in A^c$. Then $\{b\}$ is open.
Applying Lemma~\ref{lem:1} to $\{a\},\{b\}$ and $A^c\setminus\{b\}$, we may deduce that $A^c=(A^c\setminus\{b\})\cup\{b\}$ is open. So, in either case,
$A$ is closed. This shows that every subset containing $a$ is closed, which forces every subset without $a$ to be open.

\item There is exactly one open singleton in $(X,\T)$.

Similar to (1), we may deduce in this case that a nonempty subset is open if and only if it takes the open singleton as a subset (i.e., if and only if
it is a {\sl superset} of the open singleton).

\item There is no open singleton in $(X,\T)$, or equivalently, all singletons in $(X,\T)$ are closed, i.e., $(X,\T)$ is $T_1$.
\begin{enumerate}
\item[Claim 1.] The intersection of two nonempty open subsets is still nonempty.

    Suppose there exist two disjoint nonempty open subsets $U$ and $V$. Notice that every finite subset in a $T_1$ space is closed. Thus, $U$ and $V$ are infinite. By taking $u\in U$ and $v\in V$, we obtain four pairwise disjoint subsets $\{u\},\{v\},U\setminus\{u\}$ and $V\setminus\{v\}$, where  $\{u\}$ and $\{v\}$ are closed while $U\setminus\{u\}=U\cap\{u\}^c$ and
    $V\setminus\{v\}=V\cap\{v\}^c$ are open. This is a contradiction to Lemma~\ref{lem:2}, and consequently, Claim 1 holds.
\item[Claim 2.] If $U\in \T\setminus\{\varnothing\}$ and $U\subset V$, then $V\in \T\setminus\{\varnothing\}$.

Otherwise, $V$ is closed but not open, and hence $V\ne X$. This implies $U$ and $V^c$ are two disjoint nonempty open subsets in $X$, which is a contradiction to Claim 1.
\item[Claim 3.] $\bigcap(\T\setminus\{\varnothing\})=\varnothing$.

Since for every $x\in X,~\{x\}$ is closed, so $\{x\}^c\in\T\setminus\{\varnothing\}$. Thus, with $|X|\geq 4$\footnote{In fact, by the argument in the proof of Claim 1, we know that $X$ is infinite.} in mind, we may easily deduce that $\bigcap(\T\setminus\{\varnothing\})\subset\bigcap_{x\in X}\{x\}^c=\varnothing$.
\end{enumerate}

According to these three claims, $\T\setminus\{\varnothing\}$ is apparently a free filter. Since $(X,\T)$ is a connected door space,
for each $A\subset X$, either $A\in\T\setminus\{\varnothing\}$ or $A^c\in\T\setminus\{\varnothing\}$. Therefore, $\T\setminus\{\varnothing\}$ is an ultrafilter, and consequently, a free ultrafilter.
\end{enumerate}
\end{proof}

\begin{remark}\rm
As a generalization of hyperconnectedness\footnote{A topological space is {\sl hyperconnected} if it cannot be written as the union of two proper closed subsets (whether disjoint or non-disjoint). A space with included point topology is hyperconnected.}, many interesting topologies, such as Zariski topology\footnote{The Zariski topology is a topology chosen for algebraic varieties in algebraic geometry \cite{H1977}. The Zariski topology $\T$ on the set $\mathbb{R}$ of real numbers coincides with the finite complement topology (or cofinite topology), i.e., $\T=\{A\subset\mathbb{R}:A^c~\text{is finite or}~A=\varnothing\}$.} and excluded point topology, satisfy OCC.
\end{remark}

\section{Solutions of set equations \eqref{eq:additivity} and \eqref{eq}}
\label{sec:2}

For a nonempty set $X$,  an \emph{algebra} $\sigma(X)$ on $X$ is a family of subsets of $X$ closed under finite union and complement. A sub-family $\F\subset\sigma(X)$ is an \emph{ultrafilter with respect to} $\sigma(X)$ if: for any $U,V\in \F$, $U\cap V\in \F$; for any $U\in\F$ and $A\supset U$ with $A\in\sigma(X)$, $A\in\F$; and for any $U\in\sigma(X)$, either $U$ or $U^c$ belongs to $\F$.

For an algebra $\sigma(X)$ closed under arbitrary union, we say a topology $\T\subset \sigma(X)$ on $X$ is a \emph{connected door topology with respect to} $\sigma(X)$ if $\T$ is a
topology on $X$ such that every proper nonempty subset of $X$ in $\sigma(X)$ is either open or closed but not both.

\begin{lemma}\label{lem:set-0}
Let  $f:\sigma(X)\to\{z_1,z_2\}$ be a surjection satisfying Eq.~\eqref{eq}. Then the following statements hold:

(1) For each $A\in \sigma(X)$, $\{f(A),f(A^c)\}=\{z_1,z_2\}$.

(2) If $Y\in f^{-1}(f(\varnothing))$ and $Y\supset A\in\sigma(X)$, then $A\in f^{-1}(f(\varnothing))$. On the other hand, if $U\in f^{-1}(f(X))$ and $U\subset V\in\sigma(X)$,
then $V\in f^{-1}(f(X))$.

(3) $f^{-1}(z_1)$ and $f^{-1}(z_2)$ are both closed under finite intersection and finite union.
\end{lemma}

\begin{proof}

(1) If $f(\varnothing)=f(X)=z_i$, then $f(A)+f(A^c)=f(\varnothing)+f(X)=2z_i$ and hence $f(A)=f(A^c)$, $\forall A\in \sigma(X)$, $i=1,2$, which is a contradiction with $f(\sigma(X))=\{z_1,z_2\}$. Thus, $\{f(\varnothing),f(X)\}=\{z_1,z_2\}$. Without loss of generality, we assume that $f(X)=z_1$ and $f(\varnothing)=z_2$. It follows from $f(A)+f(A^c)=f(\varnothing)+f(X)=z_1+z_2$ that $\{f(A),f(A^c)\}=\{z_1,z_2\}$, $\forall A\in \sigma(X)$.

(2) By (1), we may assume $f(\varnothing)=z_2$ and $f(X)=z_1$. For each $Y\in f^{-1}(z_2)$ and each $A\subset Y$, by $f(A)+f(Y\setminus A)=f(\varnothing)+f(Y)=2z_2$, we have $f(A)=f(Y\setminus A)=z_2$, which means that the subsets of a set in $f^{-1}(z_2)$ are still in $f^{-1}(z_2)$.

For any $U\in f^{-1}(z_1)$ and any $V\supset U$, from $f(V)+f(U\cup V^c)=f(U)+f(X)=2z_1$, we get $f(V)=f(U\cup V^c)=z_1$, which implies that the supersets of any set in $f^{-1}(z_1)$ are still in $f^{-1}(z_1)$.

(3) According to {\color{blue}{Eq.~\eqref{eq}}}, $\forall A,B \in f^{-1}(z_i)$, $f(A\cup B)+f(A\cap B)=2z_i$, $i=1,2$. Thus, $f(A\cup B)=f(A\cap B)=z_i$, which means that $A\cup B,A\cap B\in f^{-1}(z_i)$, $i=1,2$.
\end{proof}

\begin{theorem}\label{thm:two-valued}
Suppose $z_1$ and $z_2$ are two distinct complex numbers. Let $\sigma(X)$ be an algebra on $X$. Then the following statements hold:
\begin{enumerate}[({S}1)]
\item
 $\F$ is an ultrafilter with respect to $\sigma(X)$ if and only if there exists a surjection $f:\sigma(X)\to\{z_1,z_2\}$ satisfying Eq.~\eqref{eq} such that $\F= f^{-1}(f(X))$.
\item
If $\sigma(X)$ is also closed under arbitrary union, then  $\T$ is a connected door topology with respect to $\sigma(X)$ if and only if there exists a surjection $f:\sigma(X)\to\{z_1,z_2\}$ satisfying Eq.~\eqref{eq} such that $\T=\{\varnothing\}\cup f^{-1}(f(X))$.
\item
If $\sigma(X)=\p(X)$, then $\T$ is a connected door topology on $X$ if and only if there exists a surjection $f:\p(X)\to\{z_1,z_2\}$ satisfying Eq.~\eqref{eq} such that $\T=\{\varnothing\}\cup f^{-1}(f(X))$.
\end{enumerate}
\end{theorem}

\begin{proof}
\begin{enumerate}[({S}1)]
\item
If $\F$ is an ultrafilter with respect to $\sigma(X)$, we may define $f:\sigma(X)\to \{z_1,z_2\}$ by
$f(A)=z_1$, $\forall A\in\F$; $f(B)=z_2$, $\forall B\in \sigma(X)\setminus\F$. It is easy to check that such an $f$ satisfies \eqref{eq} on $\sigma(X)$.

On the other hand, if $f:\sigma(X)\to \{z_1,z_2\}$ satisfies \eqref{eq}, then for any $A,B\in f^{-1}(f(X))$, Lemma \ref{lem:set-0} (3) leads to $A\cap B\in f^{-1}(f(X))$, and thus $A\cap B\ne\varnothing$ since $f(\varnothing)\ne f(X)$ by Lemma \ref{lem:set-0} (1). Keeping Lemma 3(2) in mind, this implies that $f^{-1}(f(X))$ is a filter. Using Lemma \ref{lem:set-0} (1) again, we have $A\in f^{-1}(f(X))$ or $A^c\in f^{-1}(f(X))$, which means that $f^{-1}(f(X))$ is further an ultrafilter with respect to $\sigma(X)$.
\item
If $\sigma(X)$ is closed under arbitrary union, then  Lemma \ref{lem:set-0} (2) implies that $\T=\{\varnothing\}\cup f^{-1}(f(X))$ is closed under arbitrary union. In consequence, combining (S1) with the knowledge on ultrafilter, $\T$ must be a connected door topology with respect to $\sigma(X)$. The converse is evident.
\item
If $\sigma(X)=\p(X)$, then the connected door space with respect to $\sigma(X)$ reduces to the classical connected door space.
\end{enumerate}
\end{proof}

\begin{proof}[Proof of Theorem \ref{thm:two-valued-add}]
Let $a,b,c\in X$ be distinct. If $0\notin\{z_1,z_2\}$, then $f(\{a\}),f(\{a,b\})$ and $f(\{a,b,c\})$ would be three distinct numbers, which contradicts our hypothesis.
In consequence, $0\in \{z_1,z_2\}$, and we may assume $z_2=0\ne z_1$.

Let $A$ and $B$ be two disjoint nonempty subsets. Then by $f(A\cup B)=f(A)+f(B)$, we have $f(A)=0$ or $f(B)=0$. Since $f$ is a surjection, we can take  $A\in f^{-1}(z_1)$. Then $B\in f^{-1}(0)$ and $A\cup B\in f^{-1}(z_1)$ for any $B\ne \varnothing$ with $B\cap A=\varnothing$. Therefore, $U\in f^{-1}(z_1)$ for any $U\supset A\in f^{-1}(z_1)$.

For $A,B\in f^{-1}(z_1)\setminus \{\varnothing\}$, if $A\cap B=\varnothing$, then $f(A\cup B)=f(A)+f(B)=2z_1\notin \{0,z_1\}$ which is a contradiction. So $A\cap B\ne\varnothing$.
If $A\cap B\in f^{-1}(0)$, then $A\cap B\not\in \{A,B\}$. Hence, $B\setminus A\ne \varnothing$, and thus,
\[f(A\cup B)=f(A)+f(B\setminus A)=f(A)+f(B)-f(B\cap A)=2z_1\notin \{0,z_1\},\]
which is also a contradiction. Accordingly, $A\cap B\in f^{-1}(z_1)$.

Up to now, we have shown that $f^{-1}(z_1)$ is a filter. For any $A\in\p(X)\setminus\{\varnothing,X\}$, $z_1=f(X)=f(A)+f(A^c)$, which deduces that $A\in f^{-1}(z_1)$ or $A^c\in f^{-1}(z_1)$. Thus, $f^{-1}(z_1)$ is an ultrafilter.
\end{proof}

\begin{remark}
All equations in the form of inclusion-exclusion principle are equivalent to \eqref{eq}.

For example, consider the following equation
$$f(A\cup B\cup C)=f(A)+f(B)+f(C)-f(A\cap B)-f(B\cap C)-f(A\cap C)+f(A\cap B\cap C).$$

Taking $C=\varnothing$, we get \eqref{eq}. On the other hand, if \eqref{eq} holds, then
\begin{align*}
f(A\cup B\cup C)&=f(A)+f(B\cup C)-f(A\cap (B\cup C))
\\&=f(A)+f(B)+f(C)-f(B\cap C)-(f(A\cap B)+f(A\cap C)-f(A\cap B\cap C)).
\end{align*}
\end{remark}

For a family $\F\subset \p(X)$, and a subset $A\subset X$, set $\F|_A:=\{F\subset A: F\in \F\}$.

\begin{proof}[Proof of Theorem \ref{thm:three-valued}]
\textbf{Part I}: Let  $f:\p(X)\to \{-1,0,1\}$ be a  surjection satisfying \eqref{eq} such that either $\{\varnothing,X\}\cup f^{-1}(1)$ or $\{\varnothing,X\}\cup f^{-1}(-1)$ is a topology. Now we determine the structures of such topologies.

If $f(\varnothing)=f(X)=1$ (or $-1$), then $f(A)+f(A^c)=2$ (or $-2$), and thus $f(A)=1$ (or $-1$) for any $A\subset X$.  If $\{f(\varnothing),f(X)\}=\{0,1\}$ (or $\{-1,0\}$), then $f(A)+f(A^c)=1$ (or $-1$), and thus $f(A)\in \{0,1\}$ (or $f(A)\in \{0,-1\}$) for any $A\subset X$. These contradict with  $f(\p(X))=\{-1,0,1\}$.

Therefore, there remain only two cases:
\begin{enumerate}
\item[Case 1.] $f(\varnothing)=f(X)=0$.

Without loss of generality, we may suppose $\{\varnothing,X\}\cup f^{-1}(1)$ is a topology on $X$.
Since $f^{-1}(1)\ne\varnothing$, we may pick $Y\in f^{-1}(1)$. Then $Y\not\in\{\varnothing,X\}$.
Eq.~\eqref{eq} gives $f(A)+f(Y\setminus A)=1$, $\forall A\subset Y$. Thus, $f(A)\in \{0,1\}$, and Theorem \ref{thm:two-valued} deduces that $(Y,\{\varnothing\}\cup f^{-1}(1)|_Y)$ is a connected door space. Similarly, for any $B\subset Y^c$, $f(B)\in\{-1,0\}$, $f^{-1}(-1)|_{Y^c}$ is an ultrafilter and $(Y^c,\{\varnothing\}\cup f^{-1}(-1)|_{Y^c})$ is a connected door space.

For any $U\in f^{-1}(1)$, $f(U\cap Y)+f(U\cap Y^c)=f(U)+f(\varnothing)=1$. Therefore, $U\cap Y\in f^{-1}(1)$ and $U\cap Y^c\in f^{-1}(0)$. Since $\{\varnothing,X\}\cup f^{-1}(1)$ is a topology on $X$, it can be verified that both $\{U\subset Y: f(U)=1\}$ and $\{V\subset Y^c: f(V)=0\}{\color{blue}{\cup\{Y^c\}}}$ are closed under arbitrary union. 
{\color{blue}{If $\forall x\in Y^c,~f(\{x\})=0$, then $Y^c$ has to be infinite since $f^{-1}(0)|_{Y^c}$ is closed under finite union. Take $a\neq b\in Y^c$. Then
$f(Y^c)=f(Y^c\setminus\{a\})+f(Y^c\setminus\{b\})-f(Y^c\setminus\{a,b\})=0$, a contradiction. Thus, there must be some $x\in Y^c$ such that $f(\{x\})=-1$.}}
Hence, $f^{-1}(-1)|_{Y^c}$ is a principal ultrafilter. 
In consequence, there exists a unique point $a\in Y^c$ with $f(a)=-1$, which implies $f(X\setminus\{a\})=1$. Accordingly, we may set $Y=X\setminus\{a\}$ and then $f^{-1}(1)|_{X\setminus\{a\}}$ is an ultrafilter on $X\setminus\{a\}$.

\item[Case 2.] $\{f(\varnothing),f(X)\}=\{1,-1\}$.

Without loss of generality, we may assume $f(X)=1$ and $f(\varnothing)=-1$.

\begin{enumerate}[{Subcase} 1.]
\item $\{\varnothing,X\}\cup f^{-1}(-1)$ is a topology on $X$.

If there exists ${\color{blue}{\{a\}}}\in f^{-1}(1)$, then for any $b\in X$ with $b\ne a$, $f(\{a\})+f(\{b\})=f(\varnothing)+f(\{a,b\})$, which implies $f(\{b\})=-1$ and $f(\{a,b\})=1$. Since $\{\varnothing,X\}\cup f^{-1}(-1)$ is a topology on $X$, for any  $B\subset X\setminus\{a\}$, $B=\cup_{b\in B}\{b\}\in f^{-1}(-1)$ and $f(B\cup\{a\})=f(B)+f(\{a\})-f(\varnothing)=1$, which contradicts with $f$ being surjective.

So for any $a\in X$, $f({\color{blue}{\{a\}}})\in\{-1,0\}$. If there exist three different elements $a,b,c\in f^{-1}(0)$, then it can be calculated that $f(\{a,b,c\})=2$ which leads to a  contradiction. If there exsits exactly one element $a\in f^{-1}(0)$, then $b\in f^{-1}(-1)$ for any $b\ne a$, and $X\setminus \{a\}=\cup_{b\in X\setminus \{a\}}\{b\}\in f^{-1}(-1)$. Accordingly, $f(X)=-1+0-(-1)=0$, which is a contradiction. Similarly, if $b\in f^{-1}(-1)$ for any $b\in X$, then $f(X)=-1$, also a contradiction. Thus, there exist exactly two elements $a,b\in f^{-1}(0)$ and for any $c\in X\setminus\{a,b\}$, $c\in f^{-1}(-1)$. Hence, there is no difficulty to check that $f^{-1}(-1)=\p(X\setminus\{a,b\})$.

\item $\{\varnothing,X\}\cup f^{-1}(1)$ is a topology on $X$.

Then for any $U\in f^{-1}(1)$ and $V\supset U$, $1+f(V\setminus U)=-1+f(V)$, which deduces that $V\in f^{-1}(1)$ and $V\setminus U\in f^{-1}(-1)$. By the hypothesis, for any $U,V\in f^{-1}(1)$, $U\cap V\in f^{-1}(1)$, which implies $U\cap V\ne \varnothing$. Thus, $f^{-1}(1)$ is a filter.

Let $A\subset X$ satisfy $f(A)=0$. Then $f(A^c)=0$. So, by Theorem \ref{thm:two-valued} and Theorem \ref{thm:classes-connected-door},  $\{\varnothing\}\cup f^{-1}(0)|_A$ and $\{\varnothing\}\cup f^{-1}(0)|_{A^c}$ are connected door topologies on $A$ and $A^c$, respectively, and $f(B)\in \{0,-1\}$ if $B\subset A$ or $B\subset A^c$.
 Since $f(U\cap A)+f(U\cap A^c)=f(U)+(-1)$, we immediately obtain that $U\in f^{-1}(1)$ if and only if $U\cap A$ and $U\cap A^c$ are both in $f^{-1}(0)$.

If there exists a singleton $\{b\}\in f^{-1}(0)$, then $X\setminus\{b\}\in f^{-1}(0)$, and thus $\{U\subset X\setminus\{b\}:f(U)=0\}$ is closed under arbitrary union and finite intersection. This deduces that
$f^{-1}(1)=\color{blue}{\{B|b\in B,~B\setminus \{b\}\in \F'\}}$ where $\F'$ is an ultrafilter on $X\setminus\{b\}$.

If for any singleton $x$, $f(x)=-1$, then $f^{-1}(0)|_A$ and $f^{-1}(0)|_{A^c}$ are free ultrafilters on $A$ and $A^c$, respectively. So, $f^{-1}(1)=\{U\cup V:U\subset A,V\subset A^c,U,V\in f^{-1}(0)\}$ is the union of two free ultrafilters.
\end{enumerate}
\end{enumerate}

In summary, if $f^{-1}(1)\cup \{\varnothing,X\}$ or $f^{-1}(-1)\cup \{\varnothing,X\}$ is a topology, then the topology must possesses Forms 1, 2 or 3.

\textbf{Part II}: Let $\T$ be a topology on $X$ possessing  Forms 1, 2 or 3. For each case, we shall construct a function $f$ satisfying  \eqref{eq}.

\begin{enumerate}[{Case} 1.]
\item
 $\T=\{\varnothing,X\}\cup \F'$, where $\F'$ is an ultrafilter on $X\setminus\{a\}$.

 Let $f:\p(X)\to \{-1,0,1\}$ be defined as
$$
f(U)=\begin{cases}
1,  & U\in \F', \\
0, & U\subset X\setminus \{a\} \text{ with }U\not\in \F' \text{ or } U =V\cup\{a\} \text{ where }V\in\F', \\
-1, & \text{ otherwise, i.e., }U =W\cup\{a\} \text{ where }W\subset X\setminus \{a\} \text{ and }W\not\in\F'.
\end{cases}
$$
Clearly, $\T=\{\varnothing,X\}\cup f^{-1}(1)$.
One can check that $f$ satisfies \eqref{eq} in detail:
\begin{enumerate}[{Subcase} 1.]
\item $A\in \F'$ and $B\in\F'$ $\Rightarrow$ $A\cap B,A\cup B\in\F'$ $\Rightarrow$ $A,B,A\cap B,A\cup B\in f^{-1}(1)$.

\item $A\in \F'$ and $B\subset X\setminus \{a\}$ with $B\not\in \F'$ $\Rightarrow$ $a\notin A\cap B\notin\F'$ and $A\cup B\in \F'$ $\Rightarrow$ $A,A\cup B\in f^{-1}(1)$ and $B,A\cap B\in f^{-1}(0)$.

\item $A\in \F'$ and $B=V\cup\{a\}$ with $V\in \F'$ $\Rightarrow$ $A\cap B=A\cap V\in \F'$ and $A\cup B=A\cup V\cup\{a\}$ $\Rightarrow$ $A,A\cap B\in f^{-1}(1)$ and $B,A\cup B\in f^{-1}(0)$.

\item $A\in \F'$ and $B=W\cup\{a\}$ with $W\not\in \F'$ and $W\subset X\setminus \{a\}$ $\Rightarrow$ $A\cap B=A\cap W\notin\F'$ and $A\cup B=A\cup W\cup\{a\}$ $\Rightarrow$ $A\in f^{-1}(1)$, $B\in f^{-1}(-1)$ and $A\cap B,A\cup B\in f^{-1}(0)$.

\item $A=V\cup\{a\}$ with $V\in \F'$ and $B\subset X\setminus \{a\}$ with $B\not\in \F'$ $\Rightarrow$ $A\cap B=V\cap B\subset B$, $A\cup B=V\cup B\cup \{a\}$ $\Rightarrow$ $A,A\cup B,B,A\cap B\in f^{-1}(0)$.

\item $A=V\cup\{a\}$ with $V\in \F'$ and $B=U\cup\{a\}$ with $U\in \F'$ $\Rightarrow$ $A\cap B=(V\cap U)\cup\{a\}$, $A\cup B=V\cup U\cup \{a\}$ $\Rightarrow$ $A,A\cup B,B,A\cap B\in f^{-1}(0)$.

\item $A=V\cup\{a\}$ with $V\in \F'$ and $B=W\cup\{a\}$ with $W\subset X\setminus \{a\}$ and $W\not\in \F'$ $\Rightarrow$ $A\cap B=(V\cap W)\cup \{a\}$ and $A\cup B=V\cup W\cup \{a\}$ $\Rightarrow$ $A,A\cup B\in f^{-1}(0)$ and $B,A\cap B\in f^{-1}(-1)$.

\item  $A,B\subset X\setminus \{a\}$ with $A,B\not\in \F'$ $\Rightarrow$ $A\cap B,A\cup B \subset X\setminus \{a\}$ and $A\cap B,A\cup B\not\in \F'$ $\Rightarrow$ $A,A\cup B,B,A\cap B\in f^{-1}(0)$.

\item $A\subset X\setminus \{a\}$ with $A\not\in \F'$  and $B=W\cup\{a\}$ with $W\subset X\setminus \{a\}$ and $W\not\in \F'$ $\Rightarrow$ $A\cap B=A\cap W\notin\F'$ and
$A\cup B=A\cup W\cup \{a\}$ $\Rightarrow$ $A,A\cap B\in f^{-1}(0)$ and $B,A\cup B\in f^{-1}(-1)$.

\item $A=V\cup\{a\}, B=W\cup\{a\}$ with $V,W\subset X\setminus \{a\}$ and $V,W\not\in \F'$ $\Rightarrow$ $A\cap B=(V\cap W)\cup \{a\}$ and $A\cup B=V\cup W\cup \{a\}$  $\Rightarrow$ $A,B,A\cap B,A\cup B\in f^{-1}(-1)$.
\end{enumerate}

\item
 $\T=\{\varnothing,X\}\cup\{\{a\}\cup F:F\in \F'\}$, where $\F'$ is an ultrafilter on $X\setminus\{a\}$.

 Let $f:\p(X)\to \{-1,0,1\}$ be defined as
$$
f(A)=\begin{cases}
1,  & a\in A\text{ and }A\setminus\{a\}\in \F', \\
0, & A\in\F'\text{ or } a\in A \text{ with }A\setminus\{a\}\not\in \F', \\
-1, & \text{ otherwise, i.e., }A\subset X\setminus\{a\}\text{ and }A\not\in\F'.
\end{cases}
$$
Clearly, $\T=\{\varnothing,X\}\cup f^{-1}(1)$ and the same detailed checking shows $f$ is a solution of \eqref{eq}.

\item $\T=\{\varnothing,X\}\cup \{F'\cup F'':F'\in \F', F''\in\F''\}$, where $\F'$ and $\F''$ are two free ultrafilters respectively on sets $Y$ and $X\setminus Y$.

Let $f:\p(X)\to \{-1,0,1\}$ be defined as
$$
f(A)=\begin{cases}
1,  &A= U\cup V \text{ with }U\in \F'\text{ and }V\in \F'', \\
0, & A=U\cup V\text{ with }U\in\F'\text{ and }V\in\p(Y^c)\setminus\F''\text{ or }U\in\p(Y)\setminus\F'\text{ and }V\in\F'', \\
-1, & \text{ otherwise, i.e., }A= U\cup V \text{ with }U\in \p(Y)\setminus\F'\text{ and }V\in \p(Y^c)\setminus\F''.
\end{cases}
$$
Clearly, $\T=\{\varnothing,X\}\cup f^{-1}(1)$ and the same detailed checking shows $f$ is a solution of \eqref{eq}.

\item $\T=\{\varnothing,X\}\cup\p( X\setminus\{a,b\})$, where $a\ne b\in X$.

Let $f:\p(X)\to \{-1,0,1\}$ be defined by
$$
f(A)=\begin{cases}
1,  & A\supset\{a,b\}, \\
0, & a\in A\text{ or }b\in A \text{ but }A\not\supset\{a,b\}, \\
-1, & \text{ otherwise, i.e., }A\subset X\setminus\{a,b\}.
\end{cases}
$$
Clearly, $\T=\{\varnothing,X\}\cup f^{-1}(-1)$ and the same detailed checking shows $f$ is a solution of \eqref{eq}.
\end{enumerate}
\end{proof}

\begin{proof}[Proof of Theorem \ref{thm:three-valued-add}]
First, we prove $0\in\{z_1,z_2,z_3\}$. Suppose the contrary, that $0\notin\{z_1,z_2,z_3\}$. Consider singleton sets in $X$.
\begin{enumerate}[{Case} 1.]
\item $\{f(\{x\}):x\in X\}$ has only one element.

We may assume that $f(\{x\})=z_1$ for each $x\in X$. Then $f(\{x,y\})=2z_1\in\{z_2,z_3\}$, $f(\{x,y,z\})=3z_1\in\{z_2,z_3\}$ and $f(\{x,y,z,w\})=4z_1\in\{z_2,z_3\}$ for pairwise distinct points $x,y,z$ and $w$, which leads to a contradiction with $z_1\ne 0$.

\item $\{f(\{x\}):x\in X\}$ is a set with two elements.

Suppose $f(\{x\})=f(\{y\})=z_1$ and $f(\{z\})=z_2$. Then we have $f(\{x,z\})=z_1+z_2\in \{z_1,z_2,z_3\}$. Since $z_1z_2\ne0$, one has $z_1+z_2=z_3$ and thus $f(\{x,y,z\})=2z_1+z_2=z_1+z_3\in \{z_1,z_2,z_3\}$. The same reason as above gives $z_1+z_3=z_2$. So, $(z_1+z_2)+(z_1+z_3)=z_3+z_2$, which implies $z_1=0$, a contradiction.

\item $\{f(\{x\}):x\in X\}$ has three elements.

Suppose $f(\{x\})=z_1$, $f(\{y\})=z_2$ and $f(\{z\})=z_3$. Then $f(\{x,y\})=z_1+z_2\in \{z_1,z_2,z_3\}$. Since $z_1z_2\ne0$, one has $z_1+z_2=z_3$. Similarly, $z_2+z_3=z_1$ and $z_3+z_1=z_2$. These derive $z_1=z_2=z_3=0$ and thus a contradiction arises.
\end{enumerate}

Hence, we have proved that $\{z_1,z_2,z_3\}$ possesses the form $\{z_1,z_2,0\}$. Now we prove $\{z_1,z_2,0\}$ has the form $\{0,z,2z\}$ or $\{-z,0,z\}$ and complete the proof of Theorem \ref{thm:three-valued-add}.

\begin{enumerate}[{Case} 1.]
\item $z_1+z_2\ne0$.

For any $A\in f^{-1}(z_1)$ and $B\in f^{-1}(z_2)$, if $A\cap B=\varnothing$, then $f(A\cup B)=f(A)+f(B)=z_1+z_2\in \{z_1,z_2,0\}$, which is impossible. So, $A\cap B\ne\varnothing$. Since $A\ne B$, we have $A\setminus B\ne\varnothing$ or $B\setminus A\ne\varnothing$. Without loss of generality, we may assume $B\setminus A\ne\varnothing$. Then $f(B)=f(B\setminus A)+f(A\cap B)$.

If $f(A\cap B)=0$, then $f(B\setminus A)=z_2$. Now $A\in f^{-1}(z_1),~B\setminus A\in f^{-1}(z_2)$, and $A\cap(B\setminus A)=\varnothing$, a contradiction.

If $f(A\cap B)=z_2$, then $f(B\setminus A)=0$ and $A\setminus B\ne\varnothing$. So $f(A)=f(A\setminus B)+f(A\cap B)$ and it  deduces that $f(A\setminus B)=f(A\cap B)=z_2$ and $z_1=2z_2$.

If $f(A\cap B)=z_1$, then similar discussions give   $z_2=2z_1$. So, $\{z_1,z_2,0\}=\{0,z,2z\}$ for some $z\in\mathbb{C}\setminus\{0\}$.

Now we prove $f(X)=2z$. Suppose the contrary, that $f(X)\in\{0,z\}$. Then for $A\in \p(X)\setminus\{\varnothing\}$ with $X\setminus A\ne\varnothing$, we have $f(X)=f(A)+f(X\setminus A)$. This derives $f(A)\in \{0,z\}$ for any $A\in \p(X)\setminus\{\varnothing\}$, which contradicts with that $f$ is a surjection. Therefore, $f(X)=2z$.

Let $A\in f^{-1}(z)$. It then follows from $f(X)=f(A)+f(A^c)$ that $A^c\in f^{-1}(z)$. For each $B\in \p(A)\setminus \{\varnothing,A\}$,  $f(B)+f(A\setminus B)=f(A)=z$, which implies that $B\in f^{-1}(z)$ or $B\in f^{-1}(0)$. The same property holds for $A^c$, i.e., $f(B)\in \{0,z\}$ for any $B\in\p(A^c)\setminus\{\varnothing\}$. According to Theorem \ref{thm:two-valued-add}, $f^{-1}(z)|_{A}$ and  $f^{-1}(z)|_{A^c}$ are two ultrafilters. So, $U\in f^{-1}(2z)$ if and only if $U\cap A\in f^{-1}(z)$ and $U\cap A^c\in f^{-1}(z)$.

\item $z_1+z_2=0$.

In this case, $\{z_1,z_2,0\}=\{-z,0,z\}$ for some $z\in\mathbb{C}\setminus\{0\}$. Let $$g(A)=\begin{cases}
f(A)/z,& \text{ if } A\ne \varnothing,\\
0,&\text{ if } A=\varnothing.
\end{cases}$$
Now we verify that $g(A)+g(B)=g(A\cap B)+g(A\cup B)$ for any $A,B\in \p(X)$.

If $A\cap B=\varnothing$ and $A,B\ne\varnothing$, then $g(A\cup B)=f(A\cup B)/z=f(A)/z+f(B)/z=g(A)+g(B)$.

If $A\cap B=\varnothing$ and $A=\varnothing$, then $g(A\cup B)=g(B)=g(B)+g(A)$. So, we have $g(A\cup B)=g(A)+g(B)$ whenever $A\cap B=\varnothing$.

Consequently, $g(A\cup B)+g(A\cap B)=g(A)+g(B\setminus A)+g(B\cap A)=g(A)+g(B)$ for any $A,B\in\p(X)$, and $g:\p(X)\to \{-1,0,1\}$ is a surjection with $g(\varnothing)=0$.

It is easy to check that $f^{-1}(z)= g^{-1}(1)$ and $f^{-1}(-z)= g^{-1}(-1)$ in virtue of Case 1 in the first part of the proof of Theorem \ref{thm:three-valued}.
Thus, the conclusion of Theorem \ref{thm:three-valued-add} could be verified with the help of Theorem \ref{thm:three-valued}.
\end{enumerate}
\end{proof}

\begin{remark}
A function $f:\p(X)\to \mathbb{R}$ is increasing if $f(A)\le f(B)$ whenever $A\subset B$.
If $f:\p(X)\to \mathbb{R}$ is an increasing function satisfying \eqref{eq}, then $f^{-1}(f(X))$ is a filter.
Here we give a verification:

For any $A,B\in f^{-1}(f(X))$, $f(A\cup B)+f(A\cap B)=f(A)+f(B)=2f(X)$. Since  $f$ is increasing, we have $f(A\cap B)\le f(A\cup B)\le f(X)$, and thus $f(A\cup B)=f(A\cap B)=f(X)$, which derives $A\cap B\in f^{-1}(f(X))$. Further, for any $U\supset A$, $f(U)\ge f(A)=f(X)$, which means $U\in f^{-1}(f(X))$. Therefore, $f^{-1}(f(X))$ is a filter.
\end{remark}

\section*{Acknowledgements}

This work was supported by grants from the National Natural Science Foundation of China
(No. 61772476). Dong Zhang was supported by grant from the project funded by China Postdoctoral
Science Foundation (No. 191170). The authors would like to thank the referee for very helpful suggestions and careful corrections to the previous version of the manuscript.

\end{document}